\documentclass[12pt,a4paper]{amsart}
\usepackage{mathrsfs}

\newtheorem{theo+}              {Theorem}           [section]
\newtheorem{prop+}  [theo+]     {Proposition}
\newtheorem{coro+}  [theo+]     {Corollary}
\newtheorem{lemm+}  [theo+]     {Lemma}
\newtheorem{exam+}  [theo+]     {Example}
\newtheorem{rema+}  [theo+]     {Remark}
\newtheorem{defi+}  [theo+]     {Definition}

\newenvironment{theorem}{\begin{theo+}}{\end{theo+}}
\newenvironment{proposition}{\begin{prop+}}{\end{prop+}}
\newenvironment{corollary}{\begin{coro+}}{\end{coro+}}

\usepackage{amsthm}
\theoremstyle{plain} \theoremstyle{remark}
\newtheorem{remark}{Remark}

\newtheorem*{ack}{\bf Acknowledgments}

\def \r{\mbox{${\mathbb R}$}}
\def\E{/\kern-1.0em \equiv }

\title{Some classifications of biharmonic hypersurfaces with constant scalar curvature}
\author{Shun Maeta$^{*}$  and Ye-Lin Ou$^{**}$ }

\address{Department of Mathematics, \newline\indent Shimane University,\newline\indent Matsue, 690-8504,\newline\indent Japan.
\newline\indent
E-mail:shun.maeta@gmail.com and maeta@riko.shimane-u.ac.jp}

\thanks{$^{*}$The first author is  supported partially by the Grant-in-Aid for Young Scientists(B), No.15K17542, Japan Society for the Promotion of Science, and  partially by JSPS Overseas Research Fellowships 2017-2019 No. 70. The work was done while he was visiting the Department of Mathematics of Texas A $\&$ M University-Commerce as a Visiting Scholar and he is grateful to the department and the university for the hospitality he had received during the visit.  \newline\indent $^{**}$The second author is supported by a grant from the Simons Foundation ($\#427231$, Ye-Lin Ou)}

\address{Department of
Mathematics,\newline\indent Texas A $\&$ M University-Commerce,
\newline\indent Commerce TX 75429,\newline\indent USA.\newline\indent
E-mail:yelin$\_$ou@tamu-commerce.edu}
\begin{document}

\title[Biharmonic hypersurfaces with constant scalar curvature ]{Some classifications of biharmonic hypersurfaces with constant scalar curvature  }
\subjclass[2010]{58E20, 53C12} \keywords{Biharmonic hypersurfaces,
Einstein Manifolds, constant scalar curvature, constant mean curvature.}
\date{08/28/2017}
 \maketitle

\section*{Abstract}
\begin{quote} We give some classifications of biharmonic hypersurfaces with constant scalar curvature. These include biharmonic Einstein hypersurfaces in  space forms, compact biharmonic hypersurfaces with constant scalar curvature in a sphere, and  some complete biharmonic hypersurfaces of constant scalar curvature in space forms and in a non-positively curved  Einstein space.  Our results provide additional cases (Theorem \ref{main2} and Proposition \ref{P28}) that supports the conjecture that a biharmonic submanifold in $S^{m+1}$ has constant mean curvature, and two more cases that support Chen's conjecture on biharmonic hypersurfaces (Corollaries \ref{Ch1}, \ref{Ch2}).
{\footnotesize }
\end{quote}

\section{Introduction} 

Biharmonic maps, as a generalization of harmonic maps,  are maps between Riemannian manifolds which are critical points of the bi-energy functional. Biharmonic submanifolds are the images of biharmonic isometric immersions and they include minimal submanifolds as special cases. As in the study of minimal submanifolds, a fundamental problem in the study of biharmonic submanifolds is to classify non-minimal biharmonic submanifolds (called proper biharmonic submanifolds) in a model space. For example, when the ambient space is a space form, we have the following conjectures which are still far from our reach.\\
\indent {\em Chen's Conjecture on biharmonic submanifolds of Euclidean space} \cite{Ch1} (see also \cite{Ch2}): A biharmonic submanifold in a Euclidean space is  minimal.\\

{\em Balmu\c{s}-Montaldo-Oniciuc Conjectures on biharmonic submanifolds of spheres} \cite{BMO1} (see also \cite{BMO2}):\\
\indent (1) A proper biharmonic submanifold of a sphere has constant mean curvature;\\
\indent (2) The only proper biharmonic hypersurface of $S^{m+1}$ is a part of $S^m(\frac{1}{\sqrt{2}})$ or $S^p(\frac{1}{\sqrt{2}})\times S^q(\frac{1}{\sqrt{2}})$ with $p+q=m,\;\; p\ne q$.\\

A lot of work related to these conjectures has been done since 2000,  some recent developments, partial results and open problems in the topics can be found in the recent survey \cite{Ou3} and the references therein.\\

For a hypersurface $\varphi:(M^{m},g)\rightarrow (N^{m+1}, h)$  in a Riemannian manifold, i.e., a codimensional one isometric immersion,
we choose a local unit normal vector field $\xi$ with respect to which, we have the second fundamental form  $B(X,Y)=b(X, Y)\xi$, where $b:TM\times TM\longrightarrow
C^{\infty}(M)$ is the function-valued second fundamental form. The
Gauss and the Weingarten formulae  read respectively,
\begin{equation}\label{GaussF}
{\tilde \nabla}_{X}Y=({\tilde \nabla}_{X}Y)^{\top}+({\tilde
\nabla}_{X}Y)^{\bot}=\nabla_{X}Y+b(X, Y)\xi,
\end{equation}
and 
\begin{equation}\label{WeingF}
{\tilde \nabla}_{X}\xi=({\tilde \nabla}_{X}\xi)^{\top}+({\tilde
\nabla}_{X}\xi)^{\bot}=-AX+\nabla^{\bot}_{X} \xi=-AX,
\end{equation}
where $A$ is the shape operator of the hypersurface with respect to the unit normal vector $\xi$.\\

We use ${\rm R}^M,\; {\rm Ric}^M, {\rm and}\;\;\;{\rm Scal}^M$ (respectively, ${\rm R}^N,\; {\rm Ric}^N, {\rm and}\;\;\;{\rm Scal}^N$) to denote the Riemannian, Ricci and the scalar curvature of $(M^m, g)$ (respectively,  of $ (N, h)$) with the following conventions:
\begin{eqnarray}\notag
{\rm R}^M(X, Y, Z, W)&=&\langle {\rm R}^M(Z, W)Y, X\rangle, \\\notag
{\rm R}^M(Z, W)Y&=&\nabla_Z\nabla_WY-\nabla_W\nabla_ZY-\nabla_{[Z,W]}Y,\\\notag
{\rm Ric}^M(X,Y)&=&\sum_{i=1}^m{\rm R}^M(X, e_i, Y, e_i)\; {\rm for\; an\; orthonormal \;base \;} \{e_i\} .
\end{eqnarray}
 Using  these, together with (\ref{GaussF}) and (\ref{WeingF}), we have the Gauss equation
\begin{equation}\label{GHyp}
{\rm R}^N(X,Y,Z,W)= {\rm R}^M(X, Y, Z, W)+b(X, W)b(Y, Z)-b(X, Z)b(Y, W),
\end{equation}
where
\begin{equation}
b(X,Y)=\langle AX, Y\rangle=\langle B(X, Y), \xi\rangle,
\end{equation}
for any $X, Y, Z, W \in TM$.\\

From (\ref{GHyp}) we can have the relationship between the Ricci
curvatures of the hypersurface and the ambient space
\begin{equation}\label{RI}
 {\rm Ric}^N(X,Y)= {\rm Ric}^M(X,Y)+\langle AX,AY\rangle
-mH\langle AX,Y\rangle+ {\rm R}^N(X,\xi,Y,\xi),
\end{equation}
 and the relationship between the scalar curvatures are
\begin{equation}\label{SCAL}
 {\rm Scal}^N= {\rm Scal}^M+|A|^2 -m^2H^2+2{\rm Ric}^N(\xi,\xi).
\end{equation}

It was proved in \cite{Ou1} that a hypersurface $\varphi:M^{m}\longrightarrow N^{m+1}$  with mean curvature vector $\eta=H\xi$ is biharmonic if and only if
\begin{equation}\label{BHEq}
\begin{cases}
\Delta H-H |A|^{2}+H{\rm
Ric}^N(\xi,\xi)=0,\\
 2A\,({\rm grad}\,H) +\frac{m}{2} {\rm grad}\, H^2
-2\, H \,({\rm Ric}^N\,(\xi))^{\top}=0,
\end{cases}
\end{equation}
where ${\rm Ric}^N : T_qN\longrightarrow T_qN$ denotes the Ricci
operator of the ambient space defined by $\langle {\rm Ric}^N\, (Z),
W\rangle={\rm Ric}^N (Z, W)$.

In particular, for biharmonic hypersurfaces in an Einstein space, we have 
\begin{corollary}\cite{Ou1}
A hypersurface  $\varphi:M^{m}\longrightarrow (N^{m+1}, h)$ of an Einstein manifold with ${\rm Ric}^N=\lambda h$ is biharmonic if and only if its mean curvature function $H$ solves the following equation
\begin{equation}\label{Ein}
\begin{cases}
\Delta H-H ( |A|^{2}-\lambda)=0,\\
 A\,({\rm grad}\,H) +\frac{m}{2} H {\rm grad}\, H
=0.
\end{cases}
\end{equation}
\end{corollary}

In this note, we give some classifications of biharmonic hypersurfaces with constant scalar curvature. These include biharmonic Einstein hypersurfaces in  space forms, compact biharmonic hypersurfaces with constant scalar curvature in a sphere, and  some complete biharmonic hypersurfaces with constant scalar curvature in space forms and in a non-positively curved  Einstein space.  Our results provide a further case (Theorem \ref{main2}) that supports Balmu\c{s}-Montaldo-Oniciuc conjecture: a biharmonic submanifold in $S^{m+1}$ has constant mean curvature, and two more cases that support Chen's conjecture on biharmonic hypersurfaces (Corollaries \ref{Ch1}, \ref{Ch2}).\\

\section{Biharmonic hypersurfaces with constant scaler curvature in Einstein spaces}\label{SSC}
 
 Recall that a Riemannian manifold $(M^m, g)$ is called an Einstein space if its Ricci curvature if proportional to the metric, i.e., ${\rm Ric}^M=\lambda g$. It is well known that every two-dimensional manifold is an Einstein space; any space form is of Einstein, and any $3$-dimensional Einstein space has to be a space form. One can also check that for $m\ge 3$, if $(M^m, g)$ is  an Einstein space with ${\rm Ric}^M=\lambda g$, then $\lambda$ has to be a constant and hence an Einstein space $(M^m, g)$ of dimension $m\ge 3$ has constant scalar curvature since ${\rm Scal}^M=m\lambda$.\\
 
 Our first result is the following theorem which gives a classification of biharmonic Einstein hypersurfaces in a space form.
\begin{theorem}\label{MT}
An Einstein hypersurface $M^m \hookrightarrow (N^{m+1}(C), h)\;( m\ge 3)$  in a space form  is biharmonic  if and only if it is minimal or $|A|^2=mC$. Furthermore, in the latter case, the hypersurface has  positive scalar curvature, i.e.,  ${\rm Scal}^M= m(m-2) C+m^2H^2>0$.
\end{theorem}
\begin{proof}
Suppose that the mean curvature function $H$ is not constant, then there exists an open neighborhood $U$ of $M$ on which $\nabla H\ne 0$.
Substituting $X=Y=\nabla H$ into (\ref{RI}) gives

\begin{equation}
{\rm Ric}^M(\nabla H, \nabla H)=(m-1)C |\nabla H|^2+ mHg(A(\nabla H), \nabla H)- g( A(\nabla H), A(\nabla H)).
\end{equation}
Substituting the second equation $A(\nabla H)=-\frac{m}{2}H\nabla H$ of the biharmonic equation for a hypersurface into the above equation we have
\begin{equation}\label{43}
{\rm Ric}^M(\nabla H, \nabla H)=\left((m-1)C-\frac{3}{4}m^2H^2\right)|\nabla H|^2.
\end{equation}
If $(M^m, g)$ is an Einstein hypersurface with ${\rm Ric}^M=\mu g$ for some constant $\mu$, then (\ref{43}) becomes
\begin{equation}\label{GD21}
\mu|\nabla H|^2=\left((m-1)C-\frac{3}{4}m^2H^2\right)|\nabla H|^2.
\end{equation}
Since $\nabla H\ne 0$ on $U\subset M$, we have $\mu=(m-1)C-\frac{3}{4}m^2H^2$, which implies $H$ is a constant on $U\subset M$ since $\mu$ is a constant. This contradicts the assumption that $\nabla H\ne 0$ on $U\subset M$. The contradiction shows that the mean curvature function $H$ has to be a constant. It follows from the first equation  $\Delta H-H ( |A|^{2}-mC)=0$ of (\ref{Ein}) that either $H=0$ the hypersurface is minimal, or $|A|^{2}=mC$, and in this case, $C>0$. Using  the scalar curvature ${\rm Scal}^N=(m+1)mC>0$ of $(N^{m+1}(C),h)$, $|A|^{2}=mC$ and Equation (\ref{SCAL})  we  have ${\rm Scal}^M= m(m-2) C+m^2H^2>0$.

Thus, we obtain the theorem.
\end{proof}

As an immediate consequence of Theorem \ref{MT}, we have
\begin{corollary}\label{Ch1}
A  biharmonic Einstein hypersurface of Euclidean space $\r^{m+1}$ or hyperbolic space $H^{m+1}$  is minimal; A proper biharmonic Einstein hypersurface in $S^{m+1}$ has constant mean curvature and $|A|^2=m$.
\end{corollary}

It was proved in \cite{Ch} (see also \cite{BMO2}) that form $m\ge 2$, if a compact biharmonic hypersurface in sphere $S^{m+1}(1)$ with  the squared norm of the second fundamental form satisfies $ |A|^2\le m$, then $|A|^2=0, {\rm or}\; |A|^2=m$ and it has constant mean curvature. Also, in a recent work
  \cite{Fu},  Fu proved that a biharmonic hypersurface with constant scalar curvature in $5$-dimensional space forms $M^5(C)$ has constant mean curvature, and later in \cite{FH}, it was proved that a biharmonic hypersurface with constant scalar curvature and at most six distinct principal curvatures in space forms $M^{m+1}(C)$ has constant mean curvature. Our next theorem shows that the same result in \cite{FH} holds when we replace the principal curvature assumption by the compactness of the hypersurface.

\begin{theorem}\label{main2}
A compact hypersurface with constant scalar curvature $(M^m,g)\hookrightarrow S^{m+1}$  in a sphere is biharmonic if and only if it is minimal, or it has nonzero constant mean curvature, and $|A|^2=m$.
\end{theorem}
\begin{proof}
 If the scalar curvature ${\rm Scal}^M$ of the hypersurface is constant, then, by Equation (\ref{SCAL}), we have $|A|^2=m^2H^2+{\rm constant}$, from which we have
\begin{equation}
\nabla|A|^2=2m^2H\nabla H.
\end{equation}
Using this, together with the first equation of (\ref{Ein}) with $\lambda=m$, we have
\begin{equation}\label{GD2}
\nabla\Delta H=\nabla[(|A|^2-m)H]=(|A|^2-m+2m^2H^2)\nabla H.
\end{equation}

On the other hand, we have the following estimate of the squared norm of the Hessian of $H$
\begin{align}\notag
|\nabla \rm d H|^2
=&\sum_{i,j=1}^m[\nabla \rm d H (e_i, e_j)]^2
\geq \sum_{i=1}^m[\nabla \rm d H (e_i, e_i)]^2
\geq \frac{1}{m}\left(\sum_{i=1}^m\nabla \rm d H (e_i, e_i)\right) ^2\\\label{GD3}
= &\frac{1}{m}(\Delta H)^2.
\end{align}

Substituting (\ref{43}), (\ref{GD2}) and (\ref{GD3}) into the  Bochner formula for $|\nabla H|^2$, we have
\begin{align}\label{DNH}
\frac{1}{2}\Delta|\nabla H|^2
=& |\nabla \rm d  H|^2+{\rm Ric}^M(\nabla H,\nabla H)+\langle \nabla H,\nabla \Delta H\rangle \\\notag
\ge& \frac{1}{m}(\Delta H)^2+(|A|^2-1+\frac{5}{4}m^2H^2)|\nabla H|^2.
\end{align}

Since $M$ is compact, we integrate both sides of  (\ref{DNH}) to have

\begin{align}\label{GD4}
 \int_M [ \frac{1}{m}(\Delta H)^2+(|A|^2-1+\frac{5}{4}m^2H^2)|\nabla H|^2] dv_g\le \int_M\frac{1}{2}\Delta|\nabla H|^2 dv_g=0.
\end{align}

Using compactness of $M$, $({\ref{GD2}})$, and the divergence theorem we have
\begin{align}\notag
\frac{1}{m}\int_M \Delta H\Delta H dv_g=&\frac{1}{m}\int_M \Delta H{\rm div}(\nabla H) dv_g
= -\frac{1}{m}\int_M \langle \nabla \Delta H ,\nabla H\rangle dv_g\\\label{GD100}
=& -\frac{1}{m}\int_M (|A|^2-m+2m^2H^2)|\nabla H|^2 dv_g,
\end{align}

%\begin{align}\notag
%\frac{1}{m}\int_M \Delta H\Delta H dv_g
%=& \frac{1}{m}\int_M \Delta H{\rm div}(\nabla H) dv_g\\\notag
%=& -\frac{1}{m}\int_M \langle \nabla \Delta H ,\nabla H\rangle dv_g\\\notag\label{GD100}
%=& -\frac{1}{m}\int_M (|A|^2-m+2m^2H^2)|\nabla H|^2 dv_g.
%\end{align}

Substituting this into (\ref{GD4}) we have
\begin{align}\label{GD5}
0\le  \int_M \left( (1-\frac{1}{m})|A|^2+\frac{m}{4}(5m-8)H^2\right)|\nabla H|^2 dv_g\le 0.
\end{align}

It follows that  $[(1-\frac{1}{m})|A|^2+\frac{m}{4}(5m-8)H^2]|\nabla H|^2=0$ from which, together with Newton's inequality $|A|^2\ge m H^2$, we have
\begin{equation}
0=[ (1-\frac{1}{m})|A|^2+\frac{m}{4}(5m-8)H^2]|\nabla H|^2\ge \frac{1}{16}(5m^2-4m-4)|\nabla H^2|^2.
\end{equation}

From this, we conclude that  $H=\rm constant$.  In the case if $H=\rm constant \ne 0$, we use the first equation of (\ref{Ein}) to have $|A|^2=m$. Thus, we obtain the theorem.
\end{proof}

\begin{remark}
(i)  Notice that our Theorem \ref{main2}, together with the results in \cite{Ch} and \cite{BMO2}, implies that for a compact biharmonic hypersurface in a sphere, if one of the data: ${\rm Scal}^M, \; H$ and $|A|^2$ is constant, then so are the other two.\\
(ii) We would like to point out that according to a classical result of Chern-Do Carmo-Kobayashi \cite{CCK}, any minimal hypersurface in $S^{m+1}$ with $|A|^2=m$ is locally the Clifford tori $S^p\times S^{m-p}$. On the other hand, a hypersurface of  $S^{m+1}$ with constant mean curvature  and $|A|^2=m$ is biharmonic. So it would be very interesting to classify hypersurfaces of  $S^{m+1}$ with constant mean curvature  and $|A|^2=m$. This would be an important case to solve the second one of Balmu\c{s}-Montaldo-Oniciuc conjectures.\\
(iii) Finally, we remark that there are infinitely many compact hypersurfaces of constant scalar curvatures in sphere. In fact, it was proved in \cite{Ej} that there exist countably many isometric immersions of $S^1\times S^{m-1}$ into a sphere $S^{m+1}$ so that $S^1\times S^{m-1}$ is a warped product of constant scalar curvature $m(m-1)$ with respect to the induced metric.\\
\end{remark}

\begin{theorem}\label{CPTE}
A compact Einstein hypersurface $M^m \hookrightarrow (N^{m+1}, h)$  in an Einstein manifold with ${\rm Ric}^N=\lambda h$  is biharmonic  if and only if it is minimal or $|A|^2=\lambda$. Furthermore, in the latter case, the hypersurface has  positive scalar curvature, i.e.,  ${\rm Scal}^M=(m-2)\lambda+m^2H^2>0$.
\end{theorem}

\begin{proof}
Suppose the hypersurface $M$ is Einstein with ${\rm Ric}^M=\mu g$, then, by $(\ref{SCAL})$, we have
$$(m+1)\lambda=m\mu+|A|^2-m^2H^2+2\lambda.$$
 Hence,
\begin{equation}\label{mu}
\mu=(1-\frac{1}{m})\lambda-\frac{1}{m}|A|^2+mH^2.
\end{equation}
It follows that
\begin{eqnarray}\label{NA1}
\langle \nabla H,\nabla \Delta H\rangle
&=& \langle \nabla H, \nabla (H|A|^2-\lambda H)\rangle\\\notag
&=& |A|^2|\nabla H|^2+H\langle \nabla H,\nabla |A|^2\rangle-\lambda |\nabla H|^2\\\notag
&=& (|A|^2-\lambda +2m^2H^2)|\nabla H|^2,
%=& (m^2-2m+3m^2H^2-R)|\nabla H|^2-H\langle \nabla H,\nabla R\rangle,
\end{eqnarray}
where the first equality was obtained by using  the first equation of $(\ref{Ein})$, and  the third equality follows from $(\ref{SCAL})$ and the fact that the scalar curvature of an Einstein hypersurface is constant.\\

Using these and a similar computation used in obtaining (\ref{DNH}) we have,
\begin{align}\notag
\frac{1}{2}\Delta|\nabla H|^2
=& |\nabla \rm d  H|^2+{\rm Ric}^M(\nabla H,\nabla H)+\langle \nabla H,\nabla \Delta H\rangle \\\notag
\ge& \frac{1}{m}(\Delta H)^2+\mu|\nabla H|^2+(|A|^2-\lambda+2m^2H^2)|\nabla H|^2\\\label{GD60}
=& \frac{1}{m}(\Delta H)^2+\Big(-\frac{1}{m}\lambda+(1-\frac{1}{m})|A|^2+m(2m+1)H^2\Big)|\nabla H|^2.
\end{align}

Similar to (\ref{GD100}),  we have 
\begin{align}\notag
\frac{1}{m}\int_M \Delta H\Delta H dv_g=& \frac{1}{m}\int_M \Delta H{\rm div}(\nabla H) dv_g
= -\frac{1}{m}\int_M \langle \nabla \Delta H ,\nabla H\rangle dv_g\\ \label{GD101}
=& -\frac{1}{m}\int_M (|A|^2-\lambda+2m^2H^2)|\nabla H|^2 dv_g.
\end{align}

 Integrating both sides of (\ref{GD60}) and using the compactness of $M$ and $(\ref{GD101})$ we obtain
\begin{align}
0 \geq & \int_M\Big(\frac{1}{m}(m-2)|A|^2+m(2m-1)H^2\Big)|\nabla H|^2dv_g\\\notag
\geq& \;\frac{1}{4}m(2m-1)\int_M|\nabla H^2|^2dv_g.
\end{align}

It follows that  $H$ is constant. If $H=0,$ then $M$ is minimal. If $H\not=0,$ by the first equation of $(\ref{Ein})$, we have $|A|^2=\lambda$.
Thus we have ${\rm Scal}^M=m\mu=(m-2)\lambda+m^2H^2>0$.
\end{proof}

For the classification of  complete biharmonic hypersurfaces, we notice that it was proved in \cite{NU}, \cite{Ma1} and \cite{Luo1} that a complete hypersurface $M^m\hookrightarrow (N^{m+1},h)$ in a manifold of non-positive Ricci curvature with mean curvature function $H\in L^p(M)$ for some $0<p<\infty$ is minimal. We will give a classification of biharmonic hypersurfaces of constant scalar  curvatures in an Einstein manifold using  a condition on the maximum rate of change of the mean curvature function.  For that purpose, we will need the following maximum principles:
\begin{theorem}[\cite{Y},~\cite{KL} etc.]\label{MP}
Let $M$ be a complete Riemannian manifold and $f$ is smooth function on $M$.
If one of the  following condition is satisfied, then $f$ is constant.

(i) $M$ has Ricci curvature bounded from below and $\Delta f\geq \varepsilon f$ (for some $\varepsilon>0$) for upper bounded function $f\geq0$.

(ii) $\Delta f\geq0$ for $f\geq 0$ and $f\in L^p(M)$ for some $1<p<\infty$.

(iii) $f \in L^1(M)$, $\Delta f\geq0$, $f\geq0$ and the Ricci curvature is bounded from below by $-c\{1+r^2(x)\}$, where $r(x)$ is the distance function on $M$.
\end{theorem}

Now we are ready to give some classifications of biharmonic hypersurfaces with constant scalar curvature in a non--positively curved Einstein manifold.

\begin{proposition}\label{main}
A complete biharmonic hypersurface $M^m\hookrightarrow (N^{m+1},h)$ of constant scalar curvature in a non-positively curved Einstein manifold $(N^{m+1}, h)$
is minimal if one of the following occurs:\\%$(i)$ The Ricci curvature of $M$ is bounded from below and $|\nabla H|$ is bounded from above.
$(a)$ $|\nabla H|\in L^p$, for some $2<p<\infty$.\\
$(b)$ $|\nabla H|\in L^2$ and the Ricci curvature is bounded from below by $-c\{1+r^2(x)\}$, where $r(x)$ is the distance function on $M$.
\end{proposition}
\begin{proof}

If the ambient space $(N^{m+1}, h)$ is a non-positively curved Einstein space with $\widetilde{\rm Ric}=\lambda h$, then from (\ref{RI}) we have
\begin{eqnarray}\notag
{\rm Ric}(\nabla H,\nabla H) &\ge&   \lambda|\nabla H|^2+mH\langle A(\nabla H),\nabla H\rangle-\langle A(\nabla H),A(\nabla H)\rangle\\\label{NA}
&=& (\lambda-\frac{3}{4}m^2H^2)|\nabla H|^2, %-\langle R(\xi,\nabla H)\nabla H,\xi\rangle,\notag
\end{eqnarray}
where the equality was obtained by using the second equation of the biharmonic hypersurface equation $(\ref{Ein})$.

Substituting (\ref{NA}) and (\ref{NA1}) into Bochner formula we have
\begin{eqnarray}\notag
\Delta|\nabla H|^2
&=& 2\{ | \nabla {\rm d} H|^2+{\rm Ric}(\nabla H,\nabla H)+\langle \nabla H,\nabla \Delta H\rangle\}\\\notag
&\ge& 2\{ {\rm Ric}(\nabla H,\nabla H)+\langle \nabla H,\nabla \Delta H\rangle \}\\\notag
&=&2[(\lambda-\frac{3}{4}m^2H^2)+(|A|^2-\lambda +2m^2H^2)] |\nabla H|^2\\\notag
&=& 2\Big(\frac{5}{4}m^2H^2+|A|^2\Big)|\nabla H|^2\ge 2\Big(\frac{5}{4}m^2H^2+mH^2\Big)|\nabla H|^2\\\label{NA3}
&=&\frac{1}{8}m(5m+4)|\nabla H^2|^2\ge 0.
\end{eqnarray}

Using  maximum principles (ii) and (iii)  in Theorem $\ref{MP}$ we have $|\nabla H|$ is constant.
Using this and  $(\ref{NA3})$ again we conclude that  $|\nabla H^2|=0$. It follows that $H$ is constant. If $H={\rm constant}\ne 0$, then,
by the first equation of $(\ref{Ein})$, we have $|A|^2-\lambda=0$. It follows that
 $|A|^2= \lambda\leq0$ since $N$ is non-positively curved. From this, we have $|A|^2=0$, which means that $M$ is totally geodesic and hence minimal, which is a contradiction. So we are left with the only conclusion that $H=0$, that is, the hypersurface is minimal. Thus, we complete the proof of the proposition.
\end{proof}

\begin{remark}
(i)  For a classification of complete biharmonic submanifolds with Ricci curvature bounded from below in a nonpositively curved manifold see \cite{Ma2}.\\
(ii) Our Theorem \ref{CPTE} and Proposition \ref{main} give some classifications of biharmonic hypersurface in an Einstein space. We refer the recent work \cite{IS1} and \cite{IS2} for some examples and classifications  of constant mean curvature proper biharmonic hypersurfaces in a special class of Einstein spaces--the compact Riemannian symmetric space with a $G$-invariant metric. Also, for some classifications of $f$-biharmonic hypersurfaces in an Einstein space, see the second author's paper  \cite{Ou2}.
\end{remark}

As a corollary of Proposition $\ref{main}$, we have an affirmative partial answer to Chen's conjecture:

\begin{corollary}\label{Ch2}
Any complete biharmonic hypersurface with constant scaler curvature and $|\nabla H|\in L^p(M)$ for some $2<p<\infty$ in an Euclidean space  is minimal.
\end{corollary}

\begin{proposition}\label{P28}
Let $M^m\hookrightarrow S^{m+1}$ be a complete biharmonic hypersurface with constant scaler curvature in a sphere with $H^2\geq \frac{2\varepsilon+4}{m(5m+4)}$  $(for~some~\varepsilon >0)$ .
If one of the following is satisfied, then the mean curvature is constant.

$(A)$ The Ricci curvature of $M$ is bounded from below and $|\nabla H|$ is bounded from above.

$(B)$ $|\nabla H|\in L^p$, for $2<p<\infty$.

$(C)$ $|\nabla H|\in L^2$ and the Ricci curvature is bounded from below by $-c\{1+r^2(x)\}$, where $r(x)$ is the distance function on $M$.
\end{proposition}
\begin{proof}
It follows from (\ref{DNH}) that
\begin{align*}
\Delta|\nabla H|^2
\geq&
2\Big(-1+\frac{5}{4}m^2H^2+|A|^2\Big)|\nabla H|^2.
\end{align*}
Using the  Newton's formula $|A|^2\geq m H^2$ for a hypersurface  we have 
\begin{align*}
\Delta|\nabla H|^2
\geq&
2\Big(-1+\frac{m}{4}(5m+4)H^2\Big)|\nabla H|^2.
\end{align*}
From this, together with the assumption that $H^2\geq \frac{2\varepsilon+4}{m(5m+4)}$, we have
\begin{align}\label{NH2}
\Delta|\nabla H|^2
\geq&
\varepsilon|\nabla H|^2.
\end{align}
Using the maximum principles (i), (ii), and (iii) in Theorem $\ref{MP}$ with $f=|\nabla H|^2$ we have obtain the proposition.
\end{proof}

%\begin{remark}
%We can get the same result under the assumption $ R \leq m^2-m-1+\frac{9}{4}m^2H^2-\varepsilon $ for some $\varepsilon>0$.
%\end{remark}
\begin{ack}
We would like to thank C. Oniciuc for some very useful comments that help to improve the manuscript.
\end{ack}

\end{document}